\documentclass[12pt]{amsart}
\usepackage{amsmath}
\usepackage{stmaryrd}
\usepackage{amssymb}
\usepackage{amsfonts}
\usepackage{amsthm}
\usepackage{xcolor}
\usepackage{enumerate}

\usepackage{lineno}
\PassOptionsToPackage{hyphens}{url}\usepackage{hyperref}

\hypersetup{
    colorlinks,
    linkcolor={red!50!black},
    citecolor={blue!50!black},
    urlcolor={blue!80!black}
} % To use like \url{www.nkvishnu.wordpress.com}

\newtheorem{theo}{Theorem}[section]
\newtheorem*{theo*}{Theorem}
\newtheorem{coro}[theo]{Corollary}
\newtheorem{lemm}[theo]{Lemma}

\newcommand{\N}{\mathbb{N}}
\newcommand{\Z}{\mathbb{Z}}
\newcommand{\R}{\mathbb{R}}

\title[Asymptotic formulae]{Some asymptotic formulae involving Cohen-Ramanujan expansions}

\begin{document}
\author[A Chandran]{Arya Chandran}
\address{Department of Mathematics, CVV Institute of Science and Technology, Chinmaya Vishwa Vidyapeeth (Deemed to be University),  Anthiyal-Onakkoor Road, Ernakulam, Kerala-686667, India}
\email{aryavinayachandran@gmail.com}
\author[K V Namboothiri]{K Vishnu Namboothiri}
\address{Department of Mathematics, 
Baby John Memorial Government College, Chavara, Kollam – 691583,
Kerala, India\\Department of Collegiate Education, Government of Kerala, India}
\email{kvnamboothiri@gmail.com}

\begin{abstract}
 Cohen-Ramanujan sum, denoted by $c_r^s(n)$, is an exponential sum  similar to the Ramanujan sum $c_r(n):=\sum\limits_{\substack{h=1\\{(h,r)=1}}}^{r}e^{\frac{2\pi i n h}{r}}$. An arithmetical function $f$ is said to admit a Cohen-Ramanujan expansion
 $ f(n):=\sum\limits_{r}\widehat{f}(r)c_r^s(n)$
if the series on the right hand side converges for suitable complex numbers $\widehat{f}(r)$.  Given two arithmetical functions $f$ and $g$ with absolutely convergent Cohen-Ramanujan expansions, we derive an asymptotic formula for the sum $\sum\limits_{\substack{n\leq N}}f(n)g(n+h)$ where $h$ is a fixed non negative integer. We also provide Cohen-Ramanujan expansions for certain functions to illustrate some of the results we prove consequently.
\end{abstract}

 \maketitle
\section{Introduction}
 Srinivasa Ramanujan introduced the exponential sum
\begin{align}
c_r(n)&=\sum\limits_{\substack{{m=1}\\(m,r)=1}}^{r}e^{\frac{2 \pi imn}{r}}\label{ramanujan_sum}
\end{align}
 in \cite{ramanujan1918certain} and used this sum to derive infinite Fourier series like expansions for several arithmetical functions in the form $\sum\limits_{r}a_rc_r(n)$. The exponential sum \eqref{ramanujan_sum} now known as the \emph{Ramanujan sum} has several interesting properties, many of which can be found in \cite{tom1976introduction},  \cite{mccarthy2012introduction}, or \cite{sivaramakrishnan1988classical}. If a Fourier series like expansion in the form $\sum\limits_{r}a_rc_r(n)$ exists for an arithmetical function, the  function is said to possess a \emph{Ramanujan-Fourier series expansion} or simply a \emph{Ramanujan expansion}. Ramanujan expansions were computed for various functions by many other authors, see, for example \cite{hardy1921note}. Various conditions for the existence of Ramanujan expansions appeared in papers like \cite{hardy1921note},  \cite{lucht2010survey} and \cite{murty2013ramanujan}.

Ramanujan sum has been generalized in many directions. In \cite{cohen1949extension}, E. Cohen gave a generalization defining
\begin{align}\label{gen-ram-sum}
c_r^s(n)&:=\sum\limits_{\substack{h=1\\{(h,r^s)_s=1}}}^{r^s}e^{\frac{2\pi i n h}{r^s}},
\end{align}
where for a positive integer $s$, integers $m,n$, not both zero, the \emph{generalized GCD} of $m$ and $n$ denoted by $(m,n)_{s}$ is the largest $l^s$ (where $l\in \N$) dividing both $m$ and $n$. If $(m,n)_s=1$, then $m$ and $n$ are said to be relatively $s$-prime to each other. When $s=1$, generalized GCD becomes the usual GCD and then the sum (\ref{gen-ram-sum}) reduces to the usual Ramanujan sum.
 We will refer to the sum \eqref{gen-ram-sum} by the name \emph{Cohen-Ramanujan sum}. Like the usual Ramanujan sum, for certain specific values of $s$ and $n$, this sum becomes equal to various well known arithmetical functions. For example, $c_r^s(0) = \Phi_s(r^s)=J_s(r)$, where $\Phi_s$ is the Klee's function and $J_s$ is the Jordan totient function. Please see the next section for the definitions of these two functions.  Cohen derived various properties of this sum and proved several results using this generalization in a series of papers \cite{cohen1949extension}, \cite{cohen1955extensionofr}, and \cite{cohen1956extension}.
 
 Another generalization of the Ramanujan sum, and a list of references to some more generalizations can be found in \cite{subba1966new}. Recently, Haukkanen and McCarthy gave another generalization in \cite{haukkanen1991sums}.

 An arithmetical function $f$ is said to admit a \emph{Cohen-Ramanujan expansion}
 \begin{align*}
 f(n):=\sum\limits_{r}\widehat{f}(r)c_r^s(n),
 \end{align*}
if the series on the right hand side converges for suitable complex numbers $\widehat{f}(r)$. A study on the Cohen-Ramanujan expansions of some arithmetical functions were conducted by these authors in \cite{chandran2023ramanujan}. Some  necessary and sufficient conditions for the existence of such expansions were also given there. To the best of our knowledge, such series expansions are not available for any other generalization of the Ramanujan sum as of now.

  In  \cite{hardy1921note},  Hardy derived a Ramanujan expansion  for the function $\Lambda_1(n): = (\phi(n)/n)\Lambda(n)$ in the form
  \begin{align}
  \Lambda_1(n) = \sum\limits_{q = 1}^{\infty} \frac{\mu(q)}{\phi(q)} c_q(n)\label{hardyexp},
  \end{align}
    where $\Lambda(n)$ is the von Mangoldt function defined as 
 \begin{align*} 
  \Lambda(n) := \begin{cases}
                \log p,\text{ if $n = p^k$ , $p$ is prime and $k$ any positive integer}\\
                0, \text{ otherwise}.
               \end{cases}
 \end{align*} and $\mu$ is the M{\"o}bius function.
 
 If $a$ is an arithmetical function with the Ramanujan expansion
 $a(n) = \sum\limits_{r = 1}^{\infty}a_rc_r(n)$,
then the Wiener–Khintchine theorem for an arithmetical function $a$ \cite[Section 3.2]{gopalakrishna2014ramanujanczech} can be stated as \begin{align*}
\lim\limits_{n \longrightarrow \infty }\frac{1}{N}\sum\limits_{n\le q N} a(n) a(n+h)= \sum\limits_{\substack{r=1}}^{\infty}a_r^2 c_r(h)
\end{align*} provided that certain convergence conditions are satisfied.
H. Gadiyar and R. Padma \cite{gadiyar1999ramanujan} showed  that Hardy-Littlewood conjecture on twin primes can be proved if the Wiener–Khintchine theorem can be applied to the Ramanujan expansion of the arithmetical function $\Lambda_1$ provided that the expansion is uniformly and absolutely convergent. They derived the formula
\begin{align*}
\sum\limits_{n \leq N} \Lambda(n) \Lambda(n+h) \sim N \sum\limits_{\substack{r = 1}}^{\infty}\frac{\mu^2(r)}{\phi(r)}c_r(h),
\end{align*} which agrees with the Hardy-Littlewood conjecture. But note that this identity is valid only if the Ramanujan expansion \eqref{hardyexp} derived by Hardy is absolutely convergent, which is not actually the case.   Taking forward these discussions, H. Gadiyar, M. Ram  Murty and R. Padma \cite{gopalakrishna2014ramanujan} derived an asymptotic formula for the sum
$\sum\limits_{n\leq N} f(n)g(n+h),$
where $f$ and $g$ are two arithmetical functions with absolutely convergent Ramanujan expansions.

 In this paper, we use the Cohen-Ramanujan expansion in the context initiated  in \cite{gopalakrishna2014ramanujan} and derive an asymptotic formula for the sum $\sum\limits_{n\leq N} f(n)g(n+h)$,
where $f$ and $g$ are arithmetical functions with absolutely convergent Cohen-Ramanujan expansions satisfying some additional conditions. Therefore, our results extend the results in \cite{gopalakrishna2014ramanujan}. We also provide some examples to demonstrate our results.

We also give a condition for the existence of the Cohen-Ramanujan expansion for $f_h(n) := f(n+h)$ when such an expansion for $f$ exists.

The main results we propose in this paper are the following.

\begin{theo}\label{theo1}
Suppose that $f$ and $g$ are two arithmetical functions with absolutely convergent Cohen-Ramanujan expansions
\begin{center}
 $f(n) = \sum\limits_{\substack{r}}\widehat{f}(r)c_r^{s}(n)$ and  $g(n) = \sum\limits_{\substack{k}}\widehat{g}(k)c_k^{s}(n)$
\end{center}
respectively. Suppose that $\sum\limits_{\substack{r,k}}\vert \widehat{f}(k)\widehat{g}(k) \vert (r^s,k^s)_s \tau_s(r^s) \tau_s(k^s)< \infty$. Then as $N$ tends to infinity, $\sum\limits_{\substack{n \leq N}} f(n)g(n)\sim N \sum\limits_{\substack{r}} \widehat{f}(r)\widehat{g}(r) \Phi_s(r^s)$.
\end{theo} 

\begin{theo}\label{con-sum}
Suppose that $f$ and $g$ are two arithmetical functions with absolutely convergent Cohen-Ramanujan expansions
\begin{center}

$f(n) = \sum\limits_{\substack{r}}\widehat{f}(r)c_r^{s}(n)$ and $g(n) = \sum\limits_{\substack{k}}\widehat{g}(k)c_k^{s}(n)$
\end{center}
respectively. Suppose that $\sum\limits_{\substack{r,k}}\vert \widehat{f}(r)\vert \vert\widehat{g}(k) \vert (r^sk^s)^{\frac{1}{2}} \tau_s(r^s) \tau_s(k^s)< \infty$. Then as $N$ tends to infinity, $\sum\limits_{\substack{n \leq N}} f(n)g(n+h)\sim N \sum\limits_{\substack{r}} \widehat{f}(r)\widehat{g}(r) c_r^s(h)$.

\end{theo}

\begin{theo}\label{th:fhexpansion}
Suppose that $f$ has an absolutely convergent Cohen-Ramanujan expansion with coefficients $\widehat{f}(q)$ satisfying $ \sum\limits_{\substack{q = 1}}^{\infty} \vert \widehat{f}(q) \vert \tau(q) < \infty$. Then $f_h$ has an absolutely convergent Cohen-Ramanujan expansion with coefficients $\dfrac{\widehat{f}(r)c_r^s(h)}{\Phi_s(r^s)}$.
\end{theo}

\section{Notations and basic results}
 Most of the notations, functions and identities we use in this paper are standard and their definitions can be found in \cite{tom1976introduction} or \cite{mccarthy2012introduction}. However, for the sake of completeness, we include the most important ones below.
As usual, $[m, n]$ denotes the \emph{lcm} of integers $m,n$. As in the case of the well known identity $(m,n)[m,n]=mn$, it is easy to verify that $ (m^s,n^s)_s [m^s,n^s]=m^sn^s$ where $(m,n)_s$ denotes the \emph{generalized GCD} of $m$ and $n$. An integer $m$ is said to be an \emph{$s^{th}$ power free integer} if for no positive integer $k>1$ exists such that  $k^s|m$.

If $(m, n)_s = 1$, then we say that $m$ and $n$ are \emph{relatively $s$-prime} to each other. The subset $N$ of a complete residue system $M (\text{mod }n^s)$ consisting of all elements of $M$ that are relatively $s$-prime to $n^s$ is called an $s$-reduced residue system (mod $n$).

For positive integers $s$ and $n$, the \emph{Jordan totient function} $J_s(n)$ is defined  to give the number of ordered sets of $s$ elements from a complete residue system (mod $n$) such that the greatest common divisor of each set is prime to $n$. The \emph{Klee's function} $\Phi_{s}$  is defined
   	to give the cardinality of the  set $\{m\in\N : 1\leq m\leq n, (m,n)_s=1\}$. 
  Note that $\Phi_1 = \varphi$, the Euler totient function. It is known that \cite[Section V.3]{sivaramakrishnan1988classical}  $n^s =  \sum\limits_{d|n}J_s(d)$ and $J_s(n) = \Phi_s(n^s)$. Hence
  \begin{align}\label{phi-reltn}
  n^s &=  \sum\limits_{d|n}\Phi_s(d^s) =\sum\limits_{d^s|n^s}\Phi_s(d^s).
  \end{align}
 By $\tau_{s}(n)$, where $s,n\in \N$, we mean the number of positive integers $l^s$ dividing $n$  where $l\in \N$. $\zeta(s)$ denotes the Riemann zeta function which is defined by the Equation $\zeta(s) = \sum\limits_{\substack{n=1}}^{\infty} \frac{1}{n^s}$ for $\Re(s) >1$.

  Let $x$ be a real number. We denote the distance from $x$ to the nearest integer by $\Vert x \Vert$. For $x \in \R$,  $\llbracket x\rrbracket$ will be used to denote the largest integer $n$ such that $n \leq x$.

The next theorem is essential in estalishing the  main results that we prove in this paper.

\begin{theo}\label{l1}\cite[Theorem 3.2]{tom1976introduction}
If $x \geq 1$ we have, $\sum\limits_{\substack{n \leq x}}\frac{1}{n} = \log x+C+O(\frac{1}{x})$, where $C$ is the Euler constant.
\end{theo}

It is easy to verify the next result.
\begin{lemm} \label{l3}
For $d,e,h, n, N$ positie integers, $\sum\limits_{\substack{n \leq N \\ d \mid n \\e \mid n}}1=\left\llbracket\frac{N}{[d,e]}\right\rrbracket$
and $\sum\limits_{\substack{n \leq N \\ d \mid n+h \\e \mid n+h}}1\leq \left\llbracket \frac{N+ h}{[d,e]}\right\rrbracket$.
\end{lemm}
 
For an arithmetical function $f$, its \emph{mean value} is defined by $M(f) := \lim\limits_{\substack{x\rightarrow \infty}}\frac{1}{x} \sum\limits_{\substack{{n \leq x}}} f(n)$, when the limit exists.  The existence of mean value is one of the sufficient conditions for the existence of Ramanujan expansions for arithmetical functions \cite{carmichael1932expansions}.

Cohen-Ramanujan sum (\ref{gen-ram-sum}) has the useful representation
$c_r^s(n)=\sum\limits_{\substack{d \mid r\\d^s\mid n}} \mu(\frac{r}{d}) d^s$ \cite{cohen1949extension}. Let $k,s \in \N$. The \emph{generalized sum of divisors function} \cite{chandran2023ramanujan} $\sigma_{k,s}(n)$ is defined  by $\sigma_{k,s}(n):=\sum\limits_{\substack{d^s \mid n\\d \in \N}}(d^s)^k$. For fixed $s, n\in\mathbb{N}$, $c_r^{s}(n)$ is bounded since  $\vert c_r^{s}(n)\vert  \leq \sum\limits_{\substack{d\mid r\\d^s\mid n}}d^s \leq \sigma_{1,s}(n).$

If $d \mid r$ and $t\mid r$, then by the orthogonality relation \cite{cohen1955extensionofr}, we have
\begin{align*}
 \frac{1}{r^s}\sum\limits_{\substack{m=1}}^{r^s}c_d^s(m)c_t^s(m) = \begin{cases}
\Phi_s(d^s), \text{ if } d =t\\0, \text{ otherwise.} \end{cases}
\end{align*}
Now we proceed to give the proofs of the theorems that we have stated in the introduction.
 
 \section{Proofs of the Main Results}
We start this section proving three inequalities involving products of Cohen-Ramanujan sums essential in the proof of our main results.
\begin{lemm}\label{lem1}
$\sum\limits_{\substack{n \leq N}} c_r^s(n)c_k^s(n) \leq N \tau_s(r^s)\tau_s(k^s)(r^s,k^s)_s$.
\end{lemm} 
\begin{proof}
\begin{align*}
\sum\limits_{\substack{n \leq N}} c_r^s(n)c_k^s(n) &
= \sum\limits_{\substack{n \leq N}} \sum\limits_{\substack{d \mid r\\d^s\mid n}}\mu\left(\frac{r}{d}\right)d^s \sum\limits_{\substack{e \mid k\\e^s\mid n}}\mu\left(\frac{k}{e}\right) e^s
\\&=\sum\limits_{\substack{d \mid r\\e\mid k}}\mu\left(\frac{r}{d}\right)\mu\left(\frac{k}{e}\right) d^s e^s \sum\limits_{\substack{n \leq N\\d^s \mid k\\e^s\mid n}}1.
\end{align*}
If we apply Lemma \ref{l3} to the last step above, we get
\begin{align*}
\sum\limits_{\substack{n \leq N}} c_r^s(n)c_k^s(n)&= \sum\limits_{\substack{d \mid r\\e\mid k}}\mu\left(\frac{r}{d}\right)\mu\left(\frac{k}{e}\right) d^s e^s \left\llbracket\frac{N}{[d^s,e^s]}\right\rrbracket
\\& \leq \sum\limits_{\substack{d \mid r\\e\mid k}}\mu\left(\frac{r}{d}\right)\mu\left(\frac{k}{e}\right)d^s e^s \frac{N}{[d^s,e^s]}
\\&  = \sum\limits_{\substack{d \mid r\\e\mid k}}\mu\left(\frac{r}{d}\right)\mu\left(\frac{k}{e}\right)(d^s,e^s)_s N \hspace{.2 cm}(\text{since } (d^s,e^s)_s[d^s,e^s] = d^s e^s)
\\& \leq N \sum\limits_{\substack{d \mid r\\e\mid k}}(d^s,e^s)_s
\end{align*}
Now we use Equation \eqref{phi-reltn} to get
\begin{align*}
 \sum\limits_{\substack{n \leq N}} c_r^s(n)c_k^s(n) &  \leq N \sum\limits_{\substack{d \mid r\\e\mid k}}\sum\limits_{\substack{\delta^s \mid (d^s,e^s)_s}}\Phi_s(\delta^s)\hspace{1cm}
\\&  \leq N \sum\limits_{\substack{d^s \mid r^s\\e^s\mid k^s}}\sum\limits_{\substack{\delta^s \mid d^s\\\delta^s \mid e^s}}\Phi_s(\delta^s).
\end{align*}

In the sum $\sum\limits_{\substack{d^s \mid r^s}}\sum\limits_{\substack{\delta^s \mid d^s}}\Phi_s(\delta^s)$, each $\delta$ is repeated $\tau(\frac{r^s}{\delta^s})$ times.
Therefore,
\begin{align*}
\sum\limits_{\substack{n \leq N}} c_r^s(n)c_k^s(n)
 &\leq    N \sum\limits_{\substack{d^s \mid r^s\\e^s\mid k^s}}\sum\limits_{\substack{\delta^s \mid (d^s,e^s)_s}}\Phi_s(\delta^s) 
 \\&\leq N \sum\limits_{\substack{\delta^s \mid (r^s,k^s)_s}}\Phi_s(\delta^s) \tau_s\left(\frac{r^s}{\delta^s}\right)\tau_s\left(\frac{k^s}{\delta^s}\right)
\\& \leq N \sum\limits_{\substack{\delta^s \mid (r^s,k^s)_s}}\Phi_s(\delta^s) \tau_s(r^s)\tau_s(k^s)
\\& = N \tau_s(r^s)\tau_s(k^s) \sum\limits_{\substack{\delta^s \mid (r^s,k^s)_s}}\Phi_s(\delta^s)
\\& = N \tau_s(r^s)\tau_s(k^s)(r^s,k^s)_s \hspace{1cm} \text{by Equation }(\ref{phi-reltn}).
\end{align*}
\end{proof}

\begin{lemm}\label{lem2}
$\sum\limits_{\substack{n \leq N}} c_r^s(n)c_k^s(n+h) =\delta_{r,k} N c_r^{s}(h)+O(r^sk^s \log r^sk^s)$. 
\end{lemm}
\begin{proof}
\begin{align*}
\sum\limits_{\substack{n \leq N}} c_r^s(n)c_k^s(n+h)
& = \sum\limits_{\substack{n \leq N}} \sum\limits_{\substack{a=1\\(a,r^s)_s=1}}^{r^s} e^{\frac{2\pi i a n}{r^s}}\sum\limits_{\substack{b=1\\(b,k^s)_s=1}}^{k^s} e^{\frac{2\pi i b( n+h)}{k^s}}
\\&= \sum\limits_{\substack{n \leq N}} \sum\limits_{\substack{a=1\\(a,r^s)_s=1}}^{r^s}\sum\limits_{\substack{b=1\\(b,k^s)_s=1}}^{k^s}e^{\frac{2\pi i bh}{k^s}}e^{2\pi i (\frac{a}{r^s}+\frac{b}{k^s})n}
\\&= \sum\limits_{\substack{n \leq N}} \sum\limits_{\substack{a=1\\(a,r^s)_s=1}}^{r^s}\sum\limits_{\substack{b=1\\(b,k^s)_s=1}}^{k^s}e^{\frac{2\pi i bh}{k^s}}e^{2\pi i (\frac{ak^s+br^s}{r^sk^s})n}.
\end{align*}
We estimate $$ \sum\limits_{\substack{n \leq N}} e^{2\pi i (\frac{ak^s+br^s}{r^sk^s})n}$$ in two cases. \\
\textbf{Case 1 :} $\frac{a}{r^s}+\frac{b}{k^s} \in \mathbb{Z}$.\\
Here
\begin{align*}
\frac{ak^s+br^s}{r^sk^s} \in \mathbb{Z}
&\Longleftrightarrow r^sk^s \mid ak^s+br^s
\\&\Longleftrightarrow ak^s+br^s=tr^sk^s,\text{ for some } t \in \mathbb{Z}
\\&\Longleftrightarrow br^s= tr^sk^s-ak^s.
\end{align*}
Since $k^s \mid r^sk^s-ak^s$ and $(b,k^s)_s=1$, we have $k \mid r$. Using similar arguments, we get $r \mid k$. Therefore $r=k$ and so by our assumption, $r^s\mid a+b$. Since $1 \leq a<r^s$ and $1 \leq b<r^s$, it follows that $r^s=a+b$.
Then $ \sum\limits_{\substack{n \leq N}} e^{2\pi i (\frac{ak^s+br^s}{r^sk^s})n}=  \sum\limits_{\substack{n \leq N}}1 = N.$ Therefore,
\begin{align*}
\sum\limits_{\substack{n \leq N}} c_r^s(n)c_k^s(n+h)
&=N c_r^s(h).
\end{align*}
\textbf{Case 2 :}  $\frac{a}{r^s}+\frac{b}{k^s} \notin \mathbb{Z}$.\\
In this case $e^{2\pi i (\frac{ak^s+br^s}{r^sk^s})n} \neq 1$. Let us write $N = (r^sk^s)Q+R$, where $0\leq R <r^sk^s$. Then,
\begin{align*}
\left\vert\sum\limits_{\substack{n \leq N}} e^{2\pi i \left(\frac{ak^s+br^s}{r^sk^s}\right)n}\right\vert 
&= \left\vert \sum\limits_{\substack{n \leq (r^sk^s)Q}} e^{2\pi i \left(\frac{ak^s+br^s}{r^sk^s}\right)n} + \sum\limits_{\substack{(r^sk^s)Q <n \leq N}} e^{2\pi i \left(\frac{ak^s+br^s}{r^sk^s}\right) n}\right\vert 
\\&= \left\vert 0+ \sum\limits_{\substack{(r^sk^s)Q <n \leq N}} e^{2\pi i \left( \frac{ak^s+br^s}{r^sk^s}\right)n} \right\vert .
\end{align*}
  In the above step, we used the fact that $e^{\frac{2 \pi i kl}{r^sk^s}}$ is a proper non-trivial root of unity and so $\sum\limits_{\substack{k=1}}^{r^sk^s} e^{\frac{2 \pi i kl}{r^sk^s}}=0$ when $l \neq 0$. Now
\begin{align*}
\left\vert \sum\limits_{\substack{n \leq N}} e^{2\pi i \left(\frac{ak^s+br^s}{r^sk^s}\right)n}\right\vert 
&= \left\vert\sum\limits_{\substack{(r^sk^s)Q <n \leq N}} e^{2\pi i \left(\frac{ak^s+br^s}{r^sk^s}\right)n} \right\vert.
\end{align*}
By \cite[Lemma 1]{korobov2013exponential}, for any $\alpha \in \R, p,q \in \Z$ (with $p>0$), we have $\left| \sum\limits_{\substack{k=q+1}}^{q+p}e^{2\pi i \alpha k} \right| \leq min\{p, \frac{1}{2\Vert \alpha\Vert}\}$
and so \begin{align*}
           \left\vert \sum\limits_{\substack{n \leq N}} e^{2\pi i \left(\frac{ak^s+br^s}{r^sk^s}\right)n}\right\vert \leq \frac{1}{\left\Vert \frac{ak^s+br^s}{r^sk^s}\right \Vert}.
          \end{align*}

Therefore,
\begin{align*}
\left\vert \sum\limits_{\substack{n \leq N}} \sum\limits_{\substack{a=1\\(a,r^s)_s}=1}^{r^s}\sum\limits_{\substack{b=1\\(b,k^s)_s}=1}^{k^s}e^{\frac{2\pi i bh}{ks}}e^{2\pi i \left(\frac{ak^s+br^s}{r^sk^s}\right)n} \right\vert
&\leq \sum\limits_{\substack{a=1\\(a,r^s)_s}=1}^{r^s}\sum\limits_{\substack{b=1\\(b,k^s)_s}=1}^{k^s} \frac{1}{\left\Vert \frac{ak^s+br^s}{r^sk^s}\right\Vert}.
\end{align*}
 Since $1 \leq a \leq r^s, (a,r^s)_s=1$ and  $1 \leq b \leq k^s, (b,k^s)_s=1$, there are $\Phi_s(r^s)$ choices for $a$ and $\Phi_s(k^s)$ choices for $b$. Also $(r^s,k^s)_s =1$.  Therfore there are atmost $\Phi_s(r^s)\Phi_s(k^s)=\Phi_s(r^sk^s)$ choices for $ak^s+br^s$. We claim that all these choices are distinct. Suppose that
 $a_1k^s+b_1r^s = a_2k^s+b_2r^s$. Then $(a_1-a_2)k^s = (b_2-b_1)r^s  \Rightarrow r^s \mid a_1-a_2$.
   Since $1 \leq a \leq r^s, a_1=a_2$. Simillarly we get $b_1=b_2$.

 Therefore $ak^s+br^s$ runs through an $s$-reduced residue system modulo $r^sk^s$. Thus
\begin{align*}
\sum\limits_{\substack{a=1\\(a,r^s)_s=1}}^{r^s}\sum\limits_{\substack{b=1\\(b,k^s)_s=1}}^{k^s} \frac{1}{\left\Vert \frac{ak^s+br^s}{r^sk^s}\right\Vert} 
& \leq \sum\limits_{\substack{ak^s+br^s=1\\(ak^s+br^s,r^sk^s)_s}=1}^{r^sk^s} \frac{1}{\left\Vert \frac{ak^s+br^s}{r^sk^s}\right\Vert}
\\& = \sum\limits_{\substack{t=1\\(t,r^sk^s)_s}=1}^{r^sk^s} \frac{1}{\left\Vert \frac{t}{r^sk^s}\right\Vert}
\\& \leq \sum\limits_{\substack{t=1\\(t,r^sk^s)_s}=1}^{r^sk^s} \frac{1}{ \frac{t}{r^sk^s}}
\\&  =r^sk^s \sum\limits_{\substack{t=1\\(t,r^sk^s)_s}=1}^{r^sk^s}\frac{1}{t}
\\& \leq r^sk^s \sum\limits_{\substack{t \leq r^sk^s}}\frac{1}{t}.
\end{align*}
By Theorem \ref{l1}, the last expression can be rewritten as
$r^sk^s(\log(r^s k^s)+C+O(\frac{1}{r^sk^s}))$ which is $ \ll r^sk^s\log(r^s k^s)$.

Let $(r^s,k^s)_s=d^s >1$. Then we may write $r^s=d^sr_1^s$ and $k^s=d^sk_1^s$ for some $r_1,k_1 \in \mathbb{Z}$ so that $(r_1^s,k_1^s)_s=1$.
So
\begin{align*}
\frac{a}{r^s}+\frac{b}{k^s} &= \frac{a}{r_1^sd^s}+\frac{b}{k_1^sd^s} = \frac{ak_1^s+br_1^s}{r_1^sk_1^sd^s}
= \frac{ak_1^s+br_1^s}{[r^s,k^s]}.
\end{align*}
 When $ 1 \leq a \leq r^s$ with $(a,r^s)_s=1$ and  $ 1 \leq b \leq k^s$ with $(b,k^s)_s=1$, the  $ak_1^s+br_1^s$ ranges over residue classes modulo $[r^s,k^s]$. Since $(r_1^s,k_1^s)=1,$ these residue classes are distinct mod $(r_1^s,k_1^s)$. Therefore, each class modulo $[r^s,k^s]$ is repeated atmost $d^s$ times. Then we get,
\begin{align*}
\left\vert \sum\limits_{\substack{n \leq N}} \sum\limits_{\substack{a=1\\(a,r^s)_s=1}}^{r^s}\sum\limits_{\substack{b=1\\(b,k^s)_s=1}}^{k^s}e^{\frac{2\pi i bh}{ks}}e^{2\pi i \left(\frac{ak^s+br^s}{r^sk^s}\right)n} \right\vert
&\leq \sum\limits_{\substack{a=1\\(a,r^s)_s=1}}^{r^s}\sum\limits_{\substack{b=1\\(b,k^s)_s=1}}^{k^s} \frac{1}{\left\Vert \frac{ak^s+br^s}{r^sk^s}\right\Vert}
\\& \leq \sum\limits_{\substack{a=1\\(a,r^s)_s=1}}^{r^s}\sum\limits_{\substack{b=1\\(b,k^s)_s=1}}^{k^s} \frac{1}{\left\Vert\frac{ak_1^s+br_1^s}{[r^s,k^s]}\right\Vert}
\\& \leq \sum\limits_{\substack{ak^s+br^s=1}}^{r^sk^s} \frac{1}{\left\Vert\frac{ak_1^s+br_1^s}{[r^s,k^s]}\right\Vert}.
\end{align*}

Since $(r^s,k^s)_s[r^s,k^s]=r^sk^s, r^s=d^sr_1^s, k^s = d^s k_1^s$   and $d^s = (r^s,k^s)_s$ we get  the last step above to be equal to

\begin{align*}
 \sum\limits_{\substack{ak^s+br^s=1}}^{r^sk^s} \frac{1}{\left\Vert\frac{ak_1^s+br_1^s}{d^sr_1^sk_1^s}\right\Vert} =  \sum\limits_{\substack{t=1}}^{r^sk^s} \frac{1}{\left\Vert\frac{t}{d^sr_1^sk_1^s}\right\Vert} \leq \sum\limits_{\substack{t=1}}^{r^sk^s} \frac{1}{\frac{t}{d^sr_1^sk_1^s}} = {d^sr_1^sk_1^s} \sum\limits_{\substack{t=1}}^{r^sk^s} \frac{1}{t}
\end{align*}
Now by  by Theorem \ref{l1}, we get the last step to be equal to
 \begin{align*}
 \\ {d^sr_1^sk_1^s} \Big( \log(r^sk^s)+C+O(r^sk^s)\Big)  \ll [r^s,k^s] \log(r^sk^s)
  \ll r^sk^s \log(r^sk^s).
\end{align*}
\end{proof}
\begin{lemm}\label{lem3}
For $h\geq 0$, $\vert \sum\limits_{\substack{n \leq N}} c_r^s(n)c_k^s(n+h)\vert \leq  N^{\frac{1}{2}} (N+ h )^{\frac{1}{2}} (r^sk^s)^{\frac{1}{2}}  \tau_s(r^s)\tau_s(k^s)$.
\end{lemm}
\begin{proof}
By the Cauchy-Schwartz inequality $\left(\sum\limits_{\substack{i=1}}^{n}u_iv_i\right)^2\leq \left(\sum\limits_{\substack{i=1}}^{n}u_i\right)^2 \left(\sum\limits_{\substack{i=1}}^{n}v_i\right)^2$, we have
\begin{align*}
{\vert \sum\limits_{\substack{n \leq N}} c_r^s(n)c_k^s(n+h)\vert }^2
& \leq \sum\limits_{\substack{n \leq N}} {\vert c_r^s(n) \vert}^2\sum\limits_{\substack{n \leq N}} {\vert c_k^s(n+h) \vert}^2.
\end{align*}

If we proceed as in the proof of Lemma \ref{lem1}, we can see that
\begin{align*}
&\sum\limits_{\substack{n \leq N}} {\vert c_r^s(n) \vert}^2 \leq N r^s(\tau_s(r^s))^2
\end{align*}
and
\begin{align*}
&\sum\limits_{\substack{n \leq N}} {\vert c_k^s(n+h) \vert}^2 \leq (N+ h ) k^s(\tau_s(k^s))^2.
\end{align*}
From this, we get
\begin{align*}
{\vert \sum\limits_{\substack{n \leq N}} c_r^s(n)c_k^s(n+h)\vert }^2& \leq N r^s \tau_s(r^s)^2 (N+ h)k^s \tau_s(k^s)^2.
\end{align*} 
and the claim follows.
\end{proof}
Now we prove the first theorem we stated in the introduction.
\begin{proof}[Proof of Theorem \ref{theo1}]
Using the Cohen-Ramanujan expansions of $f$ and $g$, we get
\begin{align*}
 \sum\limits_{\substack{n \leq N}} f(n)g(n) & =  \sum\limits_{\substack{n \leq N}}  \sum\limits_{\substack{r}}\widehat{f}(r)c_r^{s}(n) \sum\limits_{\substack{k}}\widehat{g}(k)c_k^{s}(n)\\
 &= \sum\limits_{\substack{r,k}}\widehat{f}(r) \widehat{g}(k)\sum\limits_{\substack{n \leq N}}c_r^{s}(n)c_k^{s}(n).
 \end{align*}
 Now we split the outersum over $r$ and $k$ into two sums. The first sum will be taken over $rk\leq U$ with $U$ tending to infinity to be chosen later. The second sum will be over $rk>U$. So
 \begin{align}\label{equ-4}
 \sum\limits_{\substack{n \leq N}} f(n)g(n)&= \sum\limits_{\substack{rk \leq U}}\widehat{f}(r) \widehat{g}(k)\sum\limits_{\substack{n \leq N}}c_r^{s}(n)c_k^{s}(n)\nonumber \\&+\sum\limits_{\substack{rk> U}}\widehat{f}(r) \widehat{g}(k)\sum\limits_{\substack{n \leq N}}c_r^{s}(n)c_k^{s}(n).
\end{align}
Now by Lemma \ref{lem2}, we get 
\begin{align*}
\sum\limits_{\substack{rk \leq U}}\widehat{f}(r) \widehat{g}(k)&\sum\limits_{\substack{n \leq N}}c_r^{s}(n)c_k^{s}(n)\\&= \sum\limits_{\substack{rk \leq U}}\widehat{f}(r) \widehat{g}(k)(\delta_{r,k} N c_r^{s}(0)+O(r^sk^s \log r^sk^s)\\
&=\sum\limits_{\substack{rk \leq U}}\widehat{f}(r) \widehat{g}(k)(\delta_{r,k} N \Phi_s(r^s)+O(r^sk^s \log r^sk^s)\\
&= N\sum\limits_{\substack{r^2 \leq U}}\widehat{f}(r) \widehat{g}(r)\Phi_s(r^s)+\sum\limits_{\substack{rk \leq U}}\widehat{f}(r) \widehat{g}(k)O(r^sk^s \log r^sk^s).\\
\end{align*}
Now $\vert c_r^s(n)\vert \leq \sigma_{1,s}(n)$. Therefore $\sum\limits_{\substack{r,k}}\widehat{f}(r)\widehat{g}(k)$ is absolutely convergent. Since $rk < U,$ the error term is $O(U^s \log U^s)$. Thus we get
\begin{align*}
\sum\limits_{\substack{rk \leq U}}\widehat{f}(r) \widehat{g}(k)\sum\limits_{\substack{n \leq N}}c_r^{s}(n)c_k^{s}(n)&= N \sum\limits_{\substack{r^2 \leq U}}\widehat{f}(r) \widehat{g}(r)\Phi_s(r^s)+O(U^s \log U^s).
\end{align*}
By Lemma \ref{lem1}, the second sum in Equation (\ref{equ-4}) satisfies the inequality
\begin{align*}
\sum\limits_{\substack{rk > U}}\widehat{f}(r) \widehat{g}(k)\sum\limits_{\substack{n \leq N}}c_r^{s}(n)c_k^{s}(n)
& \leq N \sum\limits_{\substack{rk > U}}\vert \widehat{f}(r)\vert  \vert\widehat{g}(k)\vert  \tau_s(r^s) \tau_s(k^s) (r^s,k^s)_s.\\
\end{align*}
Since $\sum\limits_{\substack{r,k}}\vert \widehat{f}(r)\vert  \vert\widehat{g}(k)\vert  \tau_s(r^s) \tau_s(k^s) (r^s,k^s)_s  < \infty$, we have
$$\sum\limits_{\substack{rk > U}}\widehat{f}(r) \widehat{g}(k)\sum\limits_{\substack{n \leq N}}c_r^{s}(n)c_k^{s}(n) \ll N.$$
Therefore $\sum\limits_{\substack{rk > U}}\widehat{f}(r) \widehat{g}(k)\sum\limits_{\substack{n \leq N}}c_r^{s}(n)c_k^{s}(n) = O(N)$.

Thus Equation (\ref{equ-4}) gives

\begin{align*}
 \sum\limits_{\substack{n \leq N}} f(n)g(n) & = N \sum\limits_{\substack{r^2 \leq  U}}\widehat{f}(r) \widehat{g}(r) \Phi_s(r^s)+O(U^s \log U^s)+O(N)
\end{align*}
which is equal to
\begin{align*}
 N \sum\limits_{\substack{r=1}}^{\infty}\widehat{f}(r) \widehat{g}(r) \Phi_s(r^s) - N \sum\limits_{\substack{r^2 > U}}\widehat{f}(r) \widehat{g}(r) \Phi_s(r^s) + O(U^s \log U^s).
\end{align*}

Since $\sum\limits_{\substack{r,k}}\vert \widehat{f}(r)\vert  \vert\widehat{g}(k)\vert  \tau_s(r^s) \tau_s(k^s) (r^s,k^s)_s  < \infty$, we have $$\sum\limits_{\substack{r}}\vert \widehat{f}(r)\vert  \vert\widehat{g}(r)\vert  \tau_s(r^s)^2 r^s  < \infty.$$
 Now  $\vert \widehat{f}(r)\vert \vert\widehat{g}(r)\vert \Phi_s(r^s) \leq \vert \widehat{f}(r)\vert \vert\widehat{g}(r)\vert r^s < \vert \widehat{f}(r)\vert \vert\widehat{g}(r)\vert r^s \tau_s(r^s)^2$, and therefore by comparison test, $\sum\limits_{\substack{r}}\widehat{f}(r) \widehat{g}(r) \Phi_s(r^s)$ converges absolutely.
 So $N \sum\limits_{\substack{r^2>U}}\widehat{f}(r) \widehat{g}(r) \Phi_s(r^s) \ll N$. That is $N \sum\limits_{\substack{r^2>U}}\widehat{f}(r) \widehat{g}(r) \Phi_s(r^s)= O(N)$.

 Thus,
 \begin{align*}
 \sum\limits_{\substack{n \leq N}} f(n)g(n) & = N \sum\limits_{\substack{r=1}}^{\infty}\widehat{f}(r) \widehat{g}(r) \Phi_s(r^s)+ O(U^s \log U^s)+O(N).
\end{align*}
Then we have  $\sum\limits_{\substack{n \leq N}} f(n)g(n)\sim N \sum\limits_{\substack{r}} \widehat{f}(r)\widehat{g}(r) \Phi_s(r^s)$.  
\end{proof}
We proceed to prove our second main theorem stated in the introduction.
\begin{proof}[Proof of Theorem \ref{con-sum}]
We have
\begin{align}
\sum\limits_{\substack{n \leq N}} f(n)g(n+h)&= \sum\limits_{\substack{n \leq N}} \sum\limits_{\substack{r}}\widehat{f}(r)c_r^{s}(n)\sum\limits_{\substack{k}}\widehat{g}(k)c_k^{s}(n+h)\nonumber
\\& = \sum\limits_{\substack{r,k}}\widehat{f}(r) \widehat{g}(k)\sum\limits_{\substack{n \leq N}} c_r^{s}(n)c_k^{s}(n+h)\nonumber
\\&= \sum\limits_{\substack{rk \leq U}}\widehat{f}(r) \widehat{g}(k)\sum\limits_{\substack{n \leq N}} c_r^{s}(n)c_k^{s}(n+h)+\nonumber
\\& \sum\limits_{\substack{rk > U}}\widehat{f}(r) \widehat{g}(k)\sum\limits_{\substack{n \leq N}} c_r^{s}(n)c_k^{s}(n+h)\label{eq:fn-gn-sum}.
\end{align}
Consider the first sum in the right-hand side of the above.
\begin{align*}
\sum\limits_{\substack{rk \leq U}}\widehat{f}(r) \widehat{g}(k)&\sum\limits_{\substack{n \leq N}} c_r^{s}(n)c_k^{s}(n+h)\\& = \sum\limits_{\substack{rk \leq U}}\widehat{f}(r) \widehat{g}(k)\{ \delta_{r,k}N c_r^s(h) + O(r^sk^s \log r^sk^s)\}\hspace{.2cm} \text{by Lemma }\ref{lem2}
\\& = N \sum\limits_{\substack{r^2 \leq U}}\widehat{f}(r) \widehat{g}(r)c_r^s(h)+ \sum\limits_{\substack{rk \leq U}}\widehat{f}(r) \widehat{g}(k)O(r^sk^s \log r^sk^s)
\\&= N  \sum\limits_{\substack{r^2 \leq U}}\widehat{f}(r) \widehat{g}(r)c_r^s(h)+O(U^s \log U^s).
\end{align*}
Now we consider the absolute value of the second sum in Equation \eqref{eq:fn-gn-sum}.
\begin{align*}
\vert \sum\limits_{\substack{rk > U}}\widehat{f}(r) \widehat{g}(k)&\sum\limits_{\substack{n \leq N}} c_r^{s}(n)c_k^{s}(n+h) \vert \\&\leq
\sum\limits_{\substack{rk > U}}\vert\widehat{f}(r)\vert  \vert\widehat{g}(k)\left| \sum\limits_{\substack{n \leq N}} c_r^{s}(n)c_k^{s}(n+h) \right| \\& \leq \sum\limits_{\substack{rk > U}}\vert\widehat{f}(r)\vert  \vert\widehat{g}(k)\vert N^{\frac{1}{2}}(N+\vert h \vert)^{\frac{1}{2}} (r^sk^s)^{\frac{1}{2}} \tau_s(r^s) \tau_s(k^s) \text{ by Lemma \ref{lem3}}
\\&\leq N^{\frac{1}{2}}(N+\vert h \vert)^{\frac{1}{2}} \sum\limits_{\substack{rk > U}}\vert\widehat{f}(r)\vert  \vert\widehat{g}(k)\vert (r^sk^s)^{\frac{1}{2}} \tau_s(r^s) \tau_s(k^s).
\end{align*}
Since we know that $\sum\limits_{\substack{rk>U}}\vert\widehat{f}(r)\vert  \vert\widehat{g}(k)\vert (r^sk^s)^{\frac{1}{2}} \tau_s(r^s) \tau_s(k^s) < \infty$, we get $\sum\limits_{\substack{rk > U}}\widehat{f}(r) \widehat{g}(k)\sum\limits_{\substack{n \leq N}} c_r^{s}(n)c_k^{s}(n+h) = O(N)$. Thus
\begin{align*}
\sum\limits_{\substack{n \leq N}} f(n)g(n+h)&=N  \sum\limits_{\substack{r^2 \leq U}}\widehat{f}(r) \widehat{g}(r)c_r^s(h)+O(U^s \log U^s)+O(N)
\\& = N  \sum\limits_{\substack{r^2 \leq U}}\widehat{f}(r) \widehat{g}(r)c_r^s(h)+O(U^s \log U^s)
\\& =  N  \sum\limits_{\substack{r=1 }}^{\infty}\widehat{f}(r) \widehat{g}(r)c_r^s(h)- N  \sum\limits_{\substack{r^2 > U}}\widehat{f}(r) \widehat{g}(r)c_r^s(h)+O(U^s \log U^s).
\end{align*}
Once again, using the inequality $\sum\limits_{\substack{r,k}}\vert\widehat{f}(r)\vert  \vert\widehat{g}(k)\vert (r^sk^s)^{\frac{1}{2}} \tau_s(r^s) \tau_s(k^s) < \infty$, we get $\sum\limits_{\substack{r}}\vert\widehat{f}(r)\vert  \vert\widehat{g}(r)\vert r^s \tau_s(r^s)^2 < \infty$. Now
\begin{align*}
\vert \widehat{f}(r) \widehat{g}(r)c_r^s(h) \vert \leq \vert\widehat{f}(r) \widehat{g}(r) r^s \tau_s(r^s) \vert
 \leq \vert\widehat{f}(r) \widehat{g}(r) r^s \tau_s(r^s)^2 \vert .
\end{align*}
 So by comparison test, $\sum\limits_{\substack{r}}\widehat{f}(r) \widehat{g}(k) c_r^s(h)$ converges absolutely.
 \begin{align*}
 \sum\limits_{\substack{n \leq N}} f(n)g(n+h)&=  N  \sum\limits_{\substack{r=1 }}^{\infty}\widehat{f}(r) \widehat{g}(r)c_r^s(h)+O(U^s \log U^s) +O(N)
 \\&=  N  \sum\limits_{\substack{r=1 }}^{\infty}\widehat{f}(r) \widehat{g}(r)c_r^s(h))+O(U^s \log U^s).
 \end{align*}
 Thus we get $\sum\limits_{\substack{n \leq N}} f(n)g(n+h) \sim N  \sum\limits_{\substack{r=1 }}^{\infty}\widehat{f}(r) \widehat{g}(r)c_r^s(h)$.
\end{proof}

We will now give an application of Theorem \ref{con-sum}. In \cite{chandran2023ramanujan}, these authors proved that, for $k, s\geq 1$,$\frac{\sigma_{ks}(n)}{n^{ks}} =\zeta(k+1) \sum\limits_{\substack{{r=1}}}^{\infty}\frac{c_r^{s}(n^s)}{r^{(k+1)s}}$, where $\sigma_s(n)=\sum\limits_{d|n}d^s$. A similar sum appears in the  next corollary and it is similar to \cite[Corollary 1]{murty2015error}.

\begin{coro}
For $a,b >1+1/2$, $s\geq 1$ and $h=m^s k$, where $k$ is an $s^{th}$ power free integer, we have
\begin{align*}
\sum\limits_{n\leq N} \frac{\sigma_{as}(n)}{n^{as}}\frac{\sigma_{bs}(n+h)}{(n+h)^{bs}} \sim N \frac{\zeta(a+1)\zeta(b+1)}{\zeta(a+b+2)}\sigma_{-(a+b+1)s}(m).
\end{align*} 

\end{coro}
\begin{proof}
Note that
\begin{align*}
 \sum\limits_{\substack{r,k}}\vert \widehat{f}(r)\vert \vert\widehat{g}(k) \vert (r^s,k^s)^{\frac{1}{2}} \tau_s(r^s) \tau_s(k^s) & \leq \sum\limits_{r,k} \left\vert \frac{\zeta(a+1)}{r^{(a+1)s}} \right\vert \left\vert \frac{\zeta(b+1)}{k^{(b+1)s}} \right\vert (r^sk^s)^{\frac{1}{2}}r^sk^s
 \\&=   \vert \zeta(a+1) \vert \vert \zeta(b+1)\vert  \sum\limits_{r,k} \frac{1}{(r^s)^{a-\frac{1}{2}}} \frac{1}{(k^s)^{b-\frac{1}{2}}}
 \\& < \infty.
\end{align*}
Now by Theorem \ref{con-sum}, 
\begin{align*}
\sum\limits_{n\leq N} \frac{\sigma_{as}(n)}{n^{as}}&\frac{\sigma_{bs}(n+h)}{(n+h)^{bs}}  \\&\sim N  \sum\limits_{r}\frac{\zeta(a+1)}{r^{(a+1)s}}\frac{\zeta(b+1)}{r^{(b+1)s}} c_r^s(h)
\\&= N \zeta(a+1)\zeta(b+1)\sum\limits_{r}\frac{c_r^s(h)}{r^{(a+b+2)s}}
\\&= N \zeta(a+1)\zeta(b+1)\sum\limits_{r}\frac{c_r^s(m^s)}{r^{(a+b+2)s}}\hspace{10pt} (\text{since }c_r^s(m^sk)=c_r(m^s))
\\&= N \zeta(a+1)\zeta(b+1)\frac{\sigma_{(a+b+1)s}(m)}{m^{(a+b+1)s}
}\frac{1}{\zeta(a+b+2)}
\\& = N \frac{\zeta(a+1)\zeta(b+1)}{\zeta(a+b+2)} \frac{\sum\limits_{d \mid m}d^{(a+b+1)s}}{m^{(a+b+1)s}}
\\& =  N \frac{\zeta(a+1)\zeta(b+1)}{\zeta(a+b+2)} \sum\limits_{k \mid m}k^{-(a+b+1)s}, \hspace{1cm} \text{where }m =kd
\\&=  N \frac{\zeta(a+1)\zeta(b+1)}{\zeta(a+b+2)}\sigma_{-(a+b+1)s}(m).
\end{align*} 
\end{proof}

Now we proceed to discuss the existence of Cohen-Ramanujan expansion for $f_h$ derived from $f$.
\begin{lemm}\label{lem4}
For $h \leq N$, $\sum\limits_{\substack{n \leq N}}c_r^s(n)c_k^s(n+h) \leq 2N\Phi_s(r^s)\tau(k)$.
\end{lemm}
\begin{proof}
Consider the sum
\begin{align*}
\sum\limits_{\substack{n \leq N}}c_r^s(n)c_k^s(n+h) & = \sum\limits_{\substack{n \leq N}}\sum\limits_{\substack{a = 1\\(a,r^s)_s=1}}^{r^s} e^{\frac{2 \pi i an}{r^s}}\sum\limits_{\substack{d \mid k\\d^s\mid n+h}} \mu\left(\frac{k}{d}\right) d^s
\\& = \sum\limits_{\substack{a = 1\\(a,r^s)_s=1}}^{r^s} e^{\frac{2 \pi i an}{r^s}} \sum\limits_{\substack{d \mid k}} \mu\left(\frac{k}{d}\right) d^s \sum\limits_{\substack{n \leq N \\ d^s \mid n+h}} 1
\end{align*}
\begin{align*}
\\& = \sum\limits_{\substack{a = 1\\(a,r^s)_s=1}}^{r^s} e^{\frac{2 \pi i an}{r^s}} \sum\limits_{\substack{d \mid k}} \mu\left(\frac{k}{d}\right) d^s \sum\limits_{\substack{n+h \leq N+h \\ d^s \mid n+h}} 1
\\& = \sum\limits_{\substack{a = 1\\(a,r^s)_s=1}}^{r^s} e^{\frac{2 \pi i an}{r^s}} \sum\limits_{\substack{d \mid k}} \mu\left(\frac{k}{d}\right) d^s \left\llbracket\frac{N+h}{d^s}\right\rrbracket
\\& \leq \sum\limits_{\substack{a = 1\\(a,r^s)_s=1}}^{r^s} e^{\frac{2 \pi i an}{r^s}} \sum\limits_{\substack{d \mid k}} \mu\left(\frac{k}{d}\right) d^s \frac{N+h}{d^s}
\\ & \leq (N+h) \sum\limits_{\substack{a = 1\\(a,r^s)_s=1}}^{r^s} 1 \sum\limits_{\substack{d \mid k}} 1
\\ & \leq (N+h) \Phi_s(r^s) \tau(k)
\\ & \leq 2N \Phi_s(r^s) \tau(k).
\end{align*}
\end{proof}
The orthogonality relation allows us to write down the possible candidates for the Cohen-Ramanujan coefficients of any given arithmetical function. Indeed, if $f(n)  =  \sum\limits_{\substack{q = 1}}^{\infty}\widehat{f}(q)c_q^s(n)$ , then $f(n)c_r^s(n) = \sum\limits_{\substack{q = 1}}^{\infty}\widehat{f}(q)c_q^s(n)c_r^s(n).$
So we get,
\begin{align*}
\lim_{x^s \to \infty} \frac{1}{x^s} \sum\limits_{\substack{n \leq x^s}} f(n) c_r^s(n)&= \lim_{x^s \to \infty} \frac{1}{x^s} \sum\limits_{\substack{n \leq x^s}} \sum\limits_{\substack{q = 1}}^{\infty}\widehat{f}(q)c_q^s(n) c_r^s(n)
\\&= \sum\limits_{\substack{q = 1}}^{\infty}\widehat{f}(q) \lim_{x^s \to \infty} \frac{1}{x^s} \sum\limits_{\substack{n \leq x^s}} c_q^s(n) c_r^s(n)
\\&= \widehat{f}(r)  \lim_{x^s \to \infty} \Phi_s(r^s).
\end{align*}
Therefore, $\widehat{f}(r)= \frac{M(fc_r^s)}{\Phi_s(r^s)}$,
where by $M(f)$, we mean the limit $\lim\limits_{\substack{x \to \infty}} \frac{1}{x} \sum\limits_{\substack{n \leq x}} f(n) $.

We are now ready to prove Theorem \ref{th:fhexpansion}.
\begin{proof}[Proof of Theorem \ref{th:fhexpansion}]
Suppose that $f(n) =  \sum\limits_{\substack{q = 1}}^{\infty}\widehat{f}(q)c_q^s(n)$ .  Then
\begin{align*}
 \sum\limits_{\substack{n \leq N}}f(n+h)c_r^s(n) &= \sum\limits_{\substack{n \leq N}} \sum\limits_{\substack{q = 1}}^{\infty}\widehat{f}(q)c_q^s(n+h) c_r^s(n)
 \\&= \sum\limits_{\substack{q = 1}}^{\infty}\widehat{f}(q)\sum\limits_{\substack{n \leq N}}c_r^s(n)c_q^s(n+h)
 \\&= \sum\limits_{\substack{q \leq U}}\widehat{f}(q)\sum\limits_{\substack{n \leq N}}c_r^s(n)c_q^s(n+h)+ \sum\limits_{\substack{q > U}}\widehat{f}(q)\sum\limits_{\substack{n \leq N}}c_r^s(n)c_q^s(n+h).
\end{align*}
We can now see the first term in the above summation can be estimated to
\begin{align*}
\sum\limits_{\substack{q \leq U}}\widehat{f}(q)&\sum\limits_{\substack{n \leq N}} c_r^s(n)c_q^s(n+h) \\&=\sum\limits_{\substack{q \leq U}}\widehat{f}(q)\left[\delta_{r,q} N c_r^s(h)+O(r^sq^s\log(r^sq^s))\right] \hspace{.5cm}\text{by Lemma }\ref{lem2}
\\&= Nc_r^s(h)+O(r^sU^s\log (r^sU^s)).
\end{align*}
and the second term can be estimated to
\begin{align*}
\sum\limits_{\substack{q > U}}\widehat{f}(q)\sum\limits_{\substack{n \leq N}}c_r^s(n)c_q^s(n+h)& \leq 2N \Phi_s(r^s) \sum\limits_{\substack{q > U}}\widehat{f}(q)\tau(q) \hspace{.5cm}{\text{by Lemma \ref{lem4}}}.
\end{align*}
Since $ \sum\limits_{\substack{q }} \vert \widehat{f}(q) \vert \tau(q) < \infty$, $\sum\limits_{\substack{q > U}}\widehat{f}(q)\sum\limits_{\substack{n \leq N}}c_r^s(n)c_q^s(n+h)=O(N).$
Therefore $ \sum\limits_{\substack{n \leq N}}f(n+h)c_r^s(n) = N\widehat{f}(r)c_r^s(h)+O(r^sU^s\log (r^sU^s))+O(N)$. Now
\begin{align*}
M(f_hc_r^s) &= \lim_{N \to \infty} \frac{1}{N} \sum\limits_{\substack{n \leq N}}f_h(n)c_r^s(n)
\\&=  \lim_{N \to \infty} \frac{1}{N} \sum\limits_{\substack{n \leq N}}f(n+h)c_r^s(n)
\\&=  \lim_{N \to \infty} \frac{1}{N} \left[N \widehat{f}(r)c_r^s(h)+O(r^sU^s\log (r^sU^s))+O(N)\right]
\\& = \widehat{f}(r)c_r^s(h).
\end{align*}
The orthogonality relation now gives
\begin{align*}
\widehat{f}_h(r) = \frac{M(f_hc_r^s)}{\Phi_s(r^s)} = \frac{\widehat{f}(r)c_r^s(h)}{\Phi_s(r^s)}.
\end{align*}
\end{proof}


\begin{thebibliography}{10}

\bibitem{tom1976introduction}
Tom Apostol.
\newblock {\em Introduction to Analytic Number Theory},
\newblock Springer, 1976.


\bibitem{cohen1949extension}
Eckford Cohen.
\newblock An extension of Ramanujan's sum,
\newblock {\em Duke Mathematical Journal}, 16(85-90):2, 1949.

\bibitem{cohen1955extensionofr}
Eckford Cohen.
\newblock An extension of {Ramanujan's} sum. II. additive properties,
\newblock {\em Duke Mathematical Journal}, 22(4):543--550, 1955.

\bibitem{cohen1956extension}
Eckford Cohen.
\newblock An extension of {Ramanujan's} sum. III. connections with totient
  functions,
\newblock {\em Duke Mathematical Journal}, 23(4):623--630, 1956.




\bibitem{gopalakrishna2014ramanujanczech}
H~Gopalakrishna~Gadiyar, and R~Padma.
\newblock Ramanujan-Fourier series and the conjecture D of Hardy and Littlewood,
\newblock {\em Czechoslovak Mathematical Journal}, 64(1):251--267,
  2014.


\bibitem{gopalakrishna2014ramanujan}
H~Gopalakrishna~Gadiyar, M~Ram~Murty, and R~Padma.
\newblock Ramanujan-Fourier series and a theorem of Ingham,
\newblock {\em Indian Journal of Pure and Applied Mathematics}, 45(5):691--706,
  2014.

\bibitem{hardy1921note}
GH~Hardy.
\newblock Note on Ramanujan's trigonometrical function cq (n) and certain series of arithmetical functions,
\newblock In {\em Proc. Cambridge Phil. Soc}, volume~20, pages 263--271, 1921.

\bibitem{haukkanen1991sums}
Pentti Haukkanen and Paul~J McCarthy.
\newblock Sums of values of even functions,
\newblock {\em Portugaliae mathematica}, 48(1):53--66, 1991.



\bibitem{lucht2010survey}
Lutz~G Lucht.
\newblock A survey of Ramanujan expansions,
\newblock {\em International Journal of Number Theory}, 6(08):1785--1799, 2010.

\bibitem{mccarthy2012introduction}
Paul~J McCarthy.
\newblock  Introduction to arithmetical functions,
\newblock {\em Springer Science \& Business Media}, 2012.

\bibitem{murty2013ramanujan}
M~Ram Murty.
\newblock Ramanujan series for arithmetical functions,
\newblock {\em Hardy-Ramanujan Journal}, 36, 2013.

\bibitem{murty2015error}
M~Ram Murty and Biswajyoti Saha.
\newblock On the error term in a Parseval type formula in the theory of Ramanujan expansions,
\newblock {\em Journal of Number Theory}, 156:125--134, 2015.



\bibitem{ramanujan1918certain}
Srinivasa Ramanujan.
\newblock On certain trigonometrical sums and their applications in the theory
  of numbers,
\newblock {\em Trans. Cambridge Philos. Soc}, 22(13):259--276, 1918.

\bibitem{sivaramakrishnan1988classical}
R~Sivaramakrishnan.
\newblock {\em Classical theory of arithmetic functions}, volume 126.
\newblock CRC Press, 1988.


\bibitem{subba1966new}
MV~Subba~Rao and VC~Harris.
\newblock A new generalization of Ramanujan's sum.
\newblock {\em Journal of the London Mathematical Society}, 1(1):595--604,
  1966.

\bibitem{chandran2023ramanujan}
Arya Chandran and K Vishnu Namboothiri
\newblock On a Ramanujan type expansion of arithmetical functions
\newblock {\em The Ramanujan Journal}, 62(1):177--188, Springer
 2023
 
 \bibitem{gadiyar1999ramanujan}
H Gopalkrishna Gadiyar and R Padma
\newblock Ramanujan--Fourier series, the Wiener--Khintchine formula and the distribution of prime pairs
\newblock {\em Physica A: Statistical Mechanics and its Applications}, 269(2-4):503--510, Elsevier
1999


 \bibitem{carmichael1932expansions}
 R. D. Carmichael
 \newblock Expansions of arithmetical functions in infinite series
\newblock {\em Proceedings of the London Mathematical Society},
  2(1): 1--26,Wiley Online Library 
  1932
  
\end{thebibliography}
\end{document}